\newif\ifpreprint
\preprinttrue
\ifpreprint
  \documentclass[11pt]{article}
  \usepackage[margin=1in]{geometry}
  \usepackage{authblk}
\else
  \documentclass[a4paper,fleqn]{cas-sc}
\fi

\usepackage[backend=biber,style=numeric,
  natbib=true,
  date=year,
  uniquename=false,
  maxcitenames=99,
  maxbibnames=99,
  giveninits=true,
  url=false,
  eprint=false,
  sorting=none,
  doi=true,
  isbn=false
]{biblatex}

\DeclareFieldFormat[article]{title}{#1\isdot}
\DeclareFieldFormat[book]{title}{#1\isdot}
\DeclareFieldFormat[inbook]{title}{#1\isdot}
\DeclareFieldFormat[incollection]{title}{#1\isdot}
\renewbibmacro{in:}{}
\DeclareSourcemap{
  \maps[datatype=bibtex]{
    \map{
      \step[fieldset=urldate, null]
    }
    \map[overwrite]{
      \step[fieldsource=shortjournal]
      \step[fieldset=journaltitle, origfieldval]
    }
  }
}
\AtEveryBibitem{\clearfield{note}}

\usepackage{anyfontsize}
\usepackage{zref-clever}
\usepackage{csvsimple-l3}
\usepackage{booktabs}
\usepackage{array}
\usepackage{graphicx}
\usepackage{mathtools,amsmath,amsthm,amsfonts,amssymb}
\usepackage{enumitem}

\let\svVert\Vert
\usepackage{physics2}
\usephysicsmodule{ab.legacy}
\AtBeginDocument{\let\Vert\svVert}

\newcommand{\CC}{\mathbb{C}}
\newcommand{\NN}{\mathbb{N}}
\newcommand{\ZZ}{\mathbb{Z}}
\newcommand{\RR}{\mathbb{R}}

\DeclareMathOperator{\pl}{pl}
\DeclareMathOperator{\pt}{pt}
\DeclareMathOperator{\pc}{pc}
\DeclareMathOperator{\fpc}{fpc}
\DeclareMathOperator{\signop}{sign}
\DeclareMathOperator{\gr}{GR}

\renewcommand{\mid}{~|~}

\newcommand{\coef}[1]{c^{(#1)}}

\newcommand{\Rdft}{R_{(\mathrm{T})}}
\newcommand{\rdft}{r_{(\mathrm{T})}}
\newcommand{\ndft}{n_{(\mathrm{T})}}

\newcommand{\xeval}{x_{e}}
\newcommand{\xmin}{x_{\min}}
\newcommand{\xmax}{x_{\max}}
\newcommand{\xc}{x_{c}}
\newcommand{\xh}{x_{h}}
\newcommand{\Kg}{K_{G}}

\newcommand{\dd}[1]{\,\mathrm{d}#1}
\newcommand{\ctxprec}{\texttt{ctx.prec}}
\newcommand{\dlg}{Dandelin--Lobachevsky--Graeffe}

\newtheorem{theorem}{Theorem}[section]

\newtheorem{lemma}[theorem]{Lemma}
\newtheorem{remark}[theorem]{Remark}
\newtheorem{example}[theorem]{Example}

\zcRefTypeSetup{theorem}{name-sg={Theorem},name-pl={Theorems}}
\AddToHook{env/definition/begin}{\zcsetup{countertype={theorem=definition}}}
\zcRefTypeSetup{definition}{name-sg={Definition},name-pl={Definitions}}
\AddToHook{env/lemma/begin}{\zcsetup{countertype={theorem=lemma}}}
\zcRefTypeSetup{lemma}{name-sg={Lemma},name-pl={Lemmas}}
\AddToHook{env/remark/begin}{\zcsetup{countertype={theorem=remark}}}
\zcRefTypeSetup{remark}{name-sg={Remark},name-pl={Remarks}}
\AddToHook{env/example/begin}{\zcsetup{countertype={theorem=example}}}
\zcRefTypeSetup{example}{name-sg={Example},name-pl={Examples}}

\numberwithin{equation}{section}

\begin{document}
\let\WriteBookmarks\relax
\def\floatpagepagefraction{1}
\def\textpagefraction{.001}
\ifpreprint
  \emergencystretch=1.5em
  \setlength{\belowcaptionskip}{8pt}
  \renewcommand\Affilfont{\normalsize}
  \renewcommand\Authands{ and }
  \title{Generalized Pellet's theorem for analytic matrix-valued functions}
  \author[1]{Hiromochi Itoh}
  \author[2]{Kei Matsushima}
  \author[3]{Takayuki Yamada\thanks{Corresponding author. Email: t.yamada@mech.t.u-tokyo.ac.jp}}
  \affil[1]{Department of Mechanical Engineering, Graduate School of Engineering, The University of Tokyo, Yayoi 2-11-16, Bunkyo-ku, 113-8656 Tokyo, Japan}
  \affil[2]{Graduate School of Advanced Science and Engineering, Hiroshima University, 1-7-1 Kagamiyama, Higashi-Hiroshima, 739-8521 Hiroshima, Japan}
  \affil[3]{Department of Strategic Studies, Institute of Engineering Innovation, Graduate School of Engineering, The University of Tokyo, Yayoi 2-11-16, Bunkyo-ku, 113-8656 Tokyo, Japan}
  \date{}
\fi

\maketitle

\begin{abstract}
  Generalized Pellet's theorem is extended to analytic matrix-valued
  functions using Descartes' rule of signs for infinite series.
  Cauchy's estimate is incorporated to allow numerical computation of
  the bounds.
  To address the issue that Pellet's theorem and its
  generalizations may not always guarantee the existence of bounds,
  for the case of matrix size one,
  we employ the \dlg\ method, which
  ensures the existence of the bounds for the roots and further
  improves them as the order of the transform increases.
  Without the need for any a priori information on the function,
  these theorems allow us to compute mathematically verified
  bounds for the roots in a certain region
  using an interval arithmetic implementation of the function.
  Numerical examples and an implementation are provided to
  illustrate the effectiveness of the proposed method.
\end{abstract}

\ifpreprint
\noindent\textbf{Keywords:} Pellet's theorem, Graeffe iteration, Interval arithmetic, Verified computation, Nonlinear eigenvalue problems
\fi

\section{Introduction}
\label{sec:introduction}

Nonlinear eigenvalue problems refer to the problem of finding $z \in
\CC$ such that $\det F(z) = 0$, where $F: \CC \to \CC^{l \times l}$
is a matrix-valued function.
If $l = 1$, the problem reduces to finding the roots of a scalar function.
In this work, $F$ is assumed to be analytic.

There are multiple methods to solve these problems, including
iterative methods (e.g., Newton's method) and contour
integration-based methods~\parencite{lyness_numerical_1967,
asakura_numerical_2009}.
However, to locate all roots in a certain region, in
most cases it is necessary to have a priori information about the
number and approximate location of the roots in the region.
In particular, for contour integration-based methods, the contour
must not be too close to the roots to reduce the numerical error of
the integration.
Thus, the localization of roots remains an important problem.

\textcite{bindel_localization_2013} proposed several methods to
localize nonlinear eigenvalues, including extensions of Gershgorin
disks and pseudospectral regions.
\textcite{michailidou_gershgorin_2018} developed further
Gershgorin-type sets for the case where $F$ is a matrix polynomial.
However, the computation of these sets may require checking if a
certain condition holds for all points in the region, which may be
computationally expensive.

Pellet's theorem, stated below, provides one of the most promising
methods to localize the roots of polynomials~\parencite[p.20]{rump_ten_2003}.
Its extension to polynomials with interval coefficients is
straightforward, and verified bounds for the roots may be obtained by
using interval arithmetic~\parencite{rump_ten_2003}.

\begin{theorem}[Pellet's theorem~\parencite{pellet_sur_1881}]
  \label{pellet-original}
  Let $n \in \NN$, let $p(z) = \sum_{j = 0}^n c_j z^j$ be a
  polynomial with complex coefficients with $c_0 c_n \neq 0$,
  and let $0 < k < n$.
  Then the number of positive roots of the polynomial $\pl_k \coloneq
  \sum_{j = 0}^{k-1} |c_j| x^j - |c_k| x^k + \sum_{j = k+1}^n
  |c_j| x^j$ is either zero or two, counting multiplicities.
  If it has two distinct positive roots $x_1 < x_2$, then $p$ has
  exactly $k$ roots in $\abs{z} \le x_1$ and no roots in $x_1 <
  \abs{z} < x_2$, counting multiplicities.
\end{theorem}

In this work, an annulus $x_1 < \abs{z} < x_2$ where no roots
exist, obtained by Pellet's theorem or its extensions, is called
a ``Pellet bound.''
A Pellet bound is said to be ``improved'' if the annulus widens
(i.e., $x_1$ becomes smaller, $x_2$ becomes larger, or both) because
the possible region of the roots becomes smaller.

\textcite{niu_verified_2007} proposed a method to compute verified
Pellet bounds for the roots of analytic functions using
Pellet's theorem, Cauchy's estimate and interval arithmetic.
However, the two methods are used separately and a certain
condition~\parencite[(2.5)]{niu_verified_2007} must be met for the Pellet bound
obtained by Pellet's theorem to be valid, which may not always be the case.
In this work, both methods are deeply integrated using Descartes' rule
of signs for infinite series~\parencite[\S1]{curtiss_recent_1918} to
eliminate the need to check such additional conditions.

\zcref{pellet-original} has been recently extended
by \textcites[Theorem 2.1]{bini_locating_2013}[Theorem
3.3]{melman_generalization_2013} to matrix-valued polynomials as follows.
In this work, the method is further extended to matrix-valued
analytic functions in an analogous manner.

\begin{theorem}[Generalized Pellet's
  theorem~\parencite{bini_locating_2013, melman_generalization_2013}]
  \label{pellet-matrix-original}
  Let $l \in \NN$, let $\norm{\cdot}$ be any submultiplicative
  norm on $l \times l$
  complex matrices $\CC^{l \times l}$.
  Let $n \in \NN$, let $P(z) = \sum_{j = 0}^n C_j z^j$ be a
  complex matrix-valued
  polynomial with $\norm{C_0} \norm{C_n} \neq 0$.
  Let $k \in \{1, \dots, n-1\}$.
  Assume that $C_k$ is invertible.
  Then the number of positive roots of the polynomial $\pl_k \coloneq
  \sum_{j = 0}^{k-1} \norm{C_j} x^j - \norm{C_k^{-1}}^{-1} x^k +
  \sum_{j = k+1}^n \norm{C_j} x^j$ is either zero or two,
  counting multiplicities.
  If it has two distinct positive roots $x_1 < x_2$, then
  $\det(P)$ has exactly $k l$ roots in $\abs{z} \le x_1$ and no roots in
  $x_1 < \abs{z} < x_2$, counting multiplicities.
\end{theorem}

However, two distinct positive roots of the polynomial $\pl_k$ in
\zcref{pellet-original} and \zcref{pellet-matrix-original} may
not always exist as illustrated in the following example,
limiting the applicability of the result.
The same issue also exists for the extension to analytic functions in this work.

\begin{example}
  Let $f(z) = (z - 1)(z + 2) = z^2 + z - 2$.
  While $f$ has exactly one root in $\abs{z} \le 1$ and no
  roots in $1 < \abs{z} < 2$, $\pl_1 (x) \coloneq x^2 - x + 2$
  has no positive roots and Pellet's theorem for $k = 1$ does
  not provide any
  information on the roots of $f$.
\end{example}

To address this issue, following \textcite{becker_near-optimal_2018,
melman_polynomials_2022}, we apply the
\dlg\ method~\parencite[Theorem
1]{householder_dandelin_1959}, a root-squaring method, to the
target function for the scalar (or $1 \times 1$ matrix) case.
This method not only guarantees the existence of Pellet
bounds, but also improves the Pellet bounds as the order of the
transform increases.
However, a direct extension of this preconditioning method to
matrix-valued functions is challenging.
Therefore, to guarantee the existence of Pellet bounds, the
determinant $\det F$ of the matrix-valued function $F$ is
targeted, which is a scalar function.

The remainder of the paper is organized as follows.
In \zcref{chap-generalized-pellet-infinite}, the direct extension of
Pellet's theorem for analytic matrix-valued functions is proved.
In \zcref{chap-generalized-pellet-cauchy}, the extension is further
improved by incorporating Cauchy's estimate to allow numerical
computation of Pellet bounds.
In \zcref{chap-verified-fourier-coefficients}, the known results on
verified computation of Fourier coefficients are summarized,
which allow verified computation of the bounds.
In \zcref{chap-preconditioning}, the convergence of the bounds
obtained by applying the
\dlg\ method is proved.
Finally, in \zcref{chap-numerical-examples}, numerical examples
using interval arithmetic are provided to demonstrate the
effectiveness of the proposed method.

\section{Generalized Pellet's Theorem for Analytic Matrix-Valued Functions}

\label{chap-generalized-pellet-infinite}

The multiplicity of a root $z$ of a complex-valued analytic
function $f$ is
defined as the smallest $k \in \NN$ such that $f^{(k)} (z) \neq 0$.
In a bounded closed region, an analytic function has at most a finite
number of roots, counting multiplicities, unless it is identically zero.
The identically zero case is not considered in this work.

The proofs of Pellet's theorem and its generalization use Descartes' rule
of signs, which relates the number of positive roots of a polynomial
to the number of sign changes in its coefficients.
The extension of Pellet's theorem to analytic (matrix-valued)
functions may be proved analogously using Descartes' rule of signs
for infinite series, stated below.

\begin{theorem}[Descartes' rule of signs for infinite
  series~{\parencite[\S1]{curtiss_recent_1918}}]
  \label{descartes-infinite}
  Let $\rho \in (0, \infty]$.
  Assume that a real-valued function $f(x) \coloneq \sum_{j =
  0}^\infty a_j x^j \not\equiv
  0$, $a_j \in \RR$ converges pointwise in $(0, \rho)$. Let $q
  \in \NN_0 \cup
  \{\infty\}$ be the number of roots of $f$ in $(0, \rho)$, counting
  multiplicities. Let
  \[
    m \coloneq \abs{\ab\{j \in \NN_0 \Bigg|
        \begin{aligned}
          &\text{ there exists } s \in \NN \text{ such that }
          \signop(a_j) \neq \signop(a_{j + s}) \\
          &\text{ and } a_{j + 1} = \dots =
          a_{j + s - 1} = 0 \text{ and } a_j \neq 0 \text{ and }
          a_{j + s} \neq 0
        \end{aligned}
    \}} \in \NN_0 \cup \{\infty\}
  \]
  be the number of sign changes
  between consecutive nonzero coefficients of $f$. Then $q \le m$.
\end{theorem}

Generalized Rouché's theorem, stated below, is used in the proof of
generalized Pellet's theorem and is also used in our proof.
Throughout this work, $\norm{\cdot}$ denotes a submultiplicative
matrix norm.

\begin{theorem}[Generalized Rouché's theorem~{\parencites[Theorem
  3.2]{melman_generalization_2013}[Theorem 9.2]{gohberg_fredholm_1990}}]
  \label{generalized-rouche}
  Let $\Omega \subseteq \CC$ be open and connected, let $A, B:
  \Omega \to \CC^{l \times l}$ be analytic matrix-valued
  functions, and assume that $A(z)$ is invertible for all $z$ on
  a smooth Jordan curve $\Gamma \subseteq \Omega$ with
  $\mathrm{Int}(\Gamma) \subseteq \Omega$. Assume that
  $\|A(z)^{-1} B(z)\| < 1$ for all $z \in \Gamma$. Then
  $\det(A + B)$ and $\det A$ have the same number of roots in
  $\mathrm{Int}(\Gamma)$, counting multiplicities.
\end{theorem}

Now the main theorem of this section is stated and proved.
The closed disk $D_\rho \coloneq \ab\{z \in \CC \mid \abs{z} \le
\rho \}$ is
defined for $\rho > 0$.
For $\rho = \infty$, $D_\rho$ is defined as $\CC$.
In this work, a function is said to be
analytic on a closed disk if it is analytic on an open set containing
the closed disk.
Additionally, an empty sum is defined as zero, and an empty
product is defined as one.

\begin{theorem}[Generalized Pellet's theorem for analytic functions]
  \label{generalized-pellet-infinite}
  Let $\rho \in (0, \infty]$, let
  $F: D_\rho \to \CC^{l \times l}, z \mapsto \sum_{j \in \NN_0} C_j
  z^j \not\equiv 0$ be a matrix-valued analytic function, and
  let $k \in \NN_0$.
  Assume that $C_k$ is invertible.
  Then the number of roots of the function $\pl_k (x)
  \coloneq \sum_{j \in \NN_0 \setminus \{k\}} \norm{C_j} x^j -
  \norm{C_k^{-1}}^{-1} x^k$ in $(0, \rho)$ is either zero, one
  or two, counting
  multiplicities.
  If $\pl_k$ has two distinct roots
  $x_1 < x_2$ in $(0, \rho)$, then
  $\det F$ has exactly $k l$ roots in $\abs{z} \le x_1$ and no roots
  in $x_1 < \abs{z} < x_2$, counting multiplicities.
\end{theorem}
\begin{proof}
  Let $x \in (0, \rho)$. If $\pl_k (x) < 0$, then for all $z$ with
  $\abs{z} = x$, $\pl_k
  (\abs{z}) < 0$ holds. Therefore, $\sum_{j \in \NN_0 \setminus \{k\}}
  \norm{C_j} \abs{z^j} < \norm{C_k^{-1}}^{-1} \abs{z^k}$ holds.
  The right-hand side is equal to $\norm{(C_k
  z^k)^{-1}}^{-1}$.
  By applying the triangle inequality to the left-hand
  side, $\norm{\sum_{j \in \NN_0 \setminus \{k\}} C_j z^j} <
  \norm{(C_k z^k)^{-1}}^{-1}$ holds.
  Let $A(z) \coloneq C_k z^k, B(z) \coloneq F(z) - A(z)$. Then
  $\norm{B(z)}
  < \norm{A(z)^{-1}}^{-1}$ holds; therefore $\norm{A(z)^{-1} B(z)} < 1$ holds.
  Thus \zcref{generalized-rouche} implies that $\det F = \det
  (A + B)$ and $\det A$ have the same number of roots in $\abs{z} <
  x$, counting multiplicities.
  Since $\det A$ has exactly $k l$ roots in $\abs{z} < x$
  (all at $z = 0$ when $k > 0$), $\det F$ has exactly $k l$ roots in
  $\abs{z} < x$, counting multiplicities.

  Since the $k$-th coefficient of $\pl_k$ is negative and
  the other coefficients of $\pl_k$ are nonnegative, the number of sign
  changes of the coefficients of $\pl_k$ is either zero, one or two.
  Therefore, by \zcref{descartes-infinite}, the number of roots
  of $\pl_k$ in $(0, \rho)$ is either zero, one or two, counting
  multiplicities.
  If $\pl_k$ has two distinct (and consequently simple) roots
  $x_1 < x_2$ in $(0, \rho)$,
  \zcref{descartes-infinite} implies that $\pl_k$ has two sign
  changes, meaning that there exists $j < k$ such that $\norm{C_j} >
  0$. Hence, for sufficiently small $x > 0$, $\pl_k (x) > 0$ holds.
  Therefore, $\pl_k (\abs{z}) < 0$ holds for all $z$ such that
  $x_1 < \abs{z} < x_2$.
  In this region, the number of roots of $\det F$ does not change,
  meaning that $\det F$ has no roots in $x_1 < \abs{z} < x_2$ and
  therefore has exactly $k l$ roots in $\abs{z} \le x_1$.
\end{proof}
\begin{remark}
  Similarly to the discussion in
  \textcite{melman_generalization_2013},
  better Pellet bounds may be obtained by substituting
  $C_k^{-1} F(z)$ for $F(z)$ in \zcref{generalized-pellet-infinite}
  if $l > 1$.
\end{remark}

\begin{remark}
  For cases where the number of roots in $(0, \rho)$ is zero or one,
  Pellet bounds may or may not be obtained depending on the signs of $\pl_k$.
  The discussion is omitted here since it can be easily handled
  using arguments analogous to those in \textcite{bini_locating_2013} and
  \zcref{generalized-pellet-cauchy} in the next section.
\end{remark}

\section{Generalized Pellet's theorem for analytic functions using
Cauchy's estimate}
\label{chap-generalized-pellet-cauchy}

\zcref{generalized-pellet-infinite} cannot be used for numerical
computation in general, since it is not feasible to calculate
$\norm{C_j}$ for all $j \in \NN_0$.
Therefore, Cauchy's estimate, stated below, is incorporated.
The maximum value of the norm of the function on a circle is
required, and this value may be obtained by dividing the circle into
multiple intervals and evaluating the upper bound of the interval
enclosure of $F$ on each interval, given that an interval arithmetic
implementation of $F$ is available.

\begin{theorem}[Cauchy's estimate for matrix-valued functions]
  \label{cauchy-estimate}
  Let $\rho > 0$, and let $F: D_\rho \to \CC^{l \times l},
  z \mapsto \sum_{j \in \NN_0} C_j z^j$ be an analytic
  matrix-valued function.
  Let $n \in \NN_0$, let $M_\rho$ be such that
  \[M_\rho
    \ge \min\ab(\max_{\abs{z} =
      \rho} \norm{F(z) - \sum_{j' = 0}^{n-1} C_{j'} z^{j'}},
  \max_{\abs{z} = \rho} \norm{F(z)}).\]
  Then for all $j \geq n$,
  $ \norm{C_j} \le M_\rho \rho^{-j} $ holds.
\end{theorem}
\begin{proof}
  It suffices to show the case $M_\rho \ge \max_{\abs{z} = \rho}
  \norm{F(z)}$ since $F(z)$ and $F(z) - \sum_{j' = 0}^{n - 1} C_{j'}
  z^{j'}$ have the same $j$-th coefficient. This holds since
  \begin{align*}
    \norm{C_j} &= \norm{\frac{1}{2 \pi i} \int_{\abs{z} = \rho}
    \frac{F(z)}{z^{j+1}} \dd{z}}
    = \norm{\frac{1}{2 \pi i} \int_0^{2 \pi} \frac{F(\rho \exp\ab(i
      \theta))}{\rho^{j+1} \exp\ab(i (j + 1) \theta)} i \rho
    \exp\ab(i \theta) \dd{\theta}}
    \le \frac{1}{2 \pi} \frac{M_\rho}{\rho^j} \int_0^{2 \pi} \dd{\theta}
    = \frac{M_\rho}{\rho^j}.
  \end{align*}
\end{proof}

\begin{theorem}[Generalized Pellet's theorem for analytic functions
  using Cauchy's estimate]
  \label{generalized-pellet-cauchy}
  Let $\rho > 0$, and let $F: D_\rho \to \CC^{l \times l}, z
  \mapsto \sum_{j \in \NN_0} C_j z^j \not\equiv 0$ be an
  analytic matrix-valued function.
  Let $k \in \NN_0$, let $n \in \NN$ be such that $k < n$.
  Let $M_\rho$ be such that
  \[
    M_\rho \ge \min\ab(\max_{\abs{z} =
    \rho} \norm{F(z) - \sum_{j = 0}^{n-1} C_j z^j}, \max_{\abs{z} =
  \rho} \norm{F(z)}).
  \]
  Assume that $C_k$ is invertible.
  Define the truncated Pellet polynomial $\pt_k^n (x) \coloneq
  \sum_{0 \le j <
  n,\, j \neq k} \norm{C_j} x^j - \norm{C_k^{-1}}^{-1} x^k$.
  Then the following assertions hold:
  \begin{enumerate}[label=(\arabic*)]
    \item The polynomial $\pc_k^n (x) \coloneq (1 - x/\rho) \pt_k^n (x)
      + M_\rho (x/\rho)^n$ has either zero, one or two roots in
      $(0, \rho)$, counting multiplicities.
    \item If $\pc_k^n$ has two distinct roots $x_1 < x_2$ in $(0,
      \rho)$, then $\det F$ has exactly $k l$ roots in $\abs{z} \le
      x_1$, counting multiplicities, and has no roots in $x_1 <
      \abs{z} < x_2$.
    \item If $\pc_k^n$ has one root $x_1$ in $(0, \rho)$:
      \begin{enumerate}[label=\arabic*.]
        \item If there exists $x \in (0, x_1)$ such that $\pc_k^n (x)
          < 0$, $\det F$ has $k l$ roots at $z = 0$ counting
          multiplicities (meaning that $z = 0$ is not a root if $k =
          0$), and has no roots in $0 < \abs{z} < x_1$.
        \item If there exists $x \in (x_1, \rho)$ such that $\pc_k^n
          (x) < 0$, $\det F$ has exactly $k l$ roots in $\abs{z} \le
          x_1$, counting
          multiplicities, and has no roots in $x_1 < \abs{z} < \rho$.
      \end{enumerate}
    \item If $\pc_k^n$ has no roots in $(0, \rho)$ and there exists
      $x \in (0, \rho)$ such that $\pc_k^n (x) < 0$, $\det F$
      has $k l$ roots at $z = 0$ counting multiplicities
      (meaning that $z = 0$ is not a root if $k = 0$), and has no roots
      in $0 < \abs{z} < \rho$.
  \end{enumerate}
\end{theorem}
\begin{proof}
  Let $\fpc_k^n$ be defined as
  \begin{equation}
    \fpc_k^n (x) \coloneq \pt_k^n (x) + M_\rho \sum_{j = n}^\infty (x /
    \rho)^j = \pt_k^n (x) + M_\rho \frac{(x / \rho)^n}{1 - x / \rho} =
    \frac{\pc_k^n (x)}{1 - x/\rho}.
    \label{fpc-def}
  \end{equation}
  The denominator $(1 - x/\rho)$ is positive in $(0, \rho)$.
  Therefore, $\pc_k^n$ and $\fpc_k^n$ have the same signs in $(0,
  \rho)$ and have common roots in $(0, \rho)$.
  Hence, it suffices to prove (1)--(4) with $\pc_k^n$
  replaced by $\fpc_k^n$.

  Assume $\fpc_k^n (x) < 0$. Then \zcref{cauchy-estimate} implies
  \begin{equation}
    \pl_k (x) = \pt_k^n (x) + \sum_{j = n}^\infty \norm{C_j} x^j < 0.
    \label{generalized-pellet-cauchy-sufficient}
  \end{equation}
  Therefore, by the same argument as in the proof of
  \zcref{generalized-pellet-infinite}, $\det F$ has
  exactly $k l$ roots in $\abs{z} < x$.

  Since the $k$-th coefficient of $\fpc_k^n$ is negative and the
  other coefficients of $\fpc_k^n$ are nonnegative, the
  number of sign
  changes of the coefficients of $\fpc_k^n$ is either zero, one or two.
  Therefore, by \zcref{descartes-infinite}, the number of roots
  of $\fpc_k^n$ in $(0, \rho)$ is either
  zero, one or two, counting multiplicities.
  This proves (1).

  Assume that either $\pc_k^n$ or $\fpc_k^n$ has two
  distinct roots $x_1 < x_2$ in $(0, \rho)$.
  Then \zcref{descartes-infinite} implies $\fpc_k^n$ has two
  sign changes.
  Therefore, for sufficiently small $x > 0$, $\fpc_k^n (x) > 0$ holds.
  Hence, $\fpc_k^n (\abs{z}) < 0$ holds for all $z$ such that $x_1 <
  \abs{z} < x_2$.
  In this region, the number of roots of $\det F(z)$ does not change,
  meaning that $\det F$ has no roots in $x_1 < \abs{z} < x_2$ and
  therefore has exactly $k l$ roots in $\abs{z} \le x_1$. This proves (2).
  The assertions (3) and (4) may be proved similarly to (2).
\end{proof}

How should the positive roots of $\pc_k^n$ be located?
By analogy with \textcite{melman_implementation_2014}, the following
result may be proved.

\begin{theorem}
  Let $\rho > 0$, and let $F$ be an analytic matrix-valued function.
  Let $k \in \NN_0$, $n \in \NN$, $M_\rho \ge 0$ be chosen as in
  \zcref{generalized-pellet-cauchy}.
  Let the function $\fpc_k^n$ be defined as in \zcref{fpc-def}.
  Then the function $x^{-k} \fpc_k^n (x)$ is convex in $(0,
  \rho)$ (i.e., its second derivative is nonnegative).
  Furthermore, if $M_\rho > 0$ or $C_j \neq 0$ for some $j$
  with $0 \leq j < n$ and $j \notin \ab\{k, k+1\}$, then $x^{-k}
  \fpc_k^n (x)$ is strictly convex in $(0,
  \rho)$, and its second
  derivative is positive in $(0, \rho)$.
\end{theorem}
\begin{proof}
  This holds since
  \[
    (x^{-k} \fpc_k^n (x))'' = \sum_{0 \leq j < n, j \notin \ab\{k,
    k+1\}} (j - k)(j - k - 1) \norm{C_j} x^{j - k - 2} + M_\rho
    \sum_{j = n}^\infty (j - k)(j - k - 1) \rho^{-j} x^{j - k - 2}
  \]
  and $(j - k)(j - k - 1) > 0$ for all $j \in \NN_0 \setminus \ab\{k, k+1\}$.
\end{proof}

Therefore, a point $x^* \in (0, \rho)$ close enough to the
unique minimizer of $x^{-k} \fpc_k^n
(x)$ in $(0, \rho)$ if it exists, or close enough to $0$ if $x^{-k}
\fpc_k^n (x)$ is increasing in $(0, \rho)$, or close enough to $\rho$
if $x^{-k} \fpc_k^n (x)$ is decreasing in $(0, \rho)$, can be
obtained by (convex) optimization methods.
If $\fpc_k^n (x^*) > 0$, the roots of $\pc_k^n$ in $(0, \rho)$
do not exist or cannot be located.
If $\fpc_k^n (x^*) < 0$, the roots of $\pc_k^n$ in
$(0, \rho)$ may be located by applying (convex) root-finding
methods in $(0, x^*)$ and $(x^*, \rho)$.

\section{Verified computation of Fourier coefficients}
\label{chap-verified-fourier-coefficients}

It is known that the Fourier coefficients of an analytic
function may be approximated by the discrete Fourier transform with
the following error estimate.
Hence, given that an interval arithmetic implementation of an
analytic function is available, the interval enclosures of its Fourier
coefficients may be calculated.
Moreover, by applying this result to each entry of a
matrix-valued analytic function, the interval enclosures of its
Fourier coefficients may be obtained.

\begin{theorem}[Error estimate for discrete Fourier coefficients]
  \label{ck-error}
  Let $\Rdft > \rdft > 0$, let $f: D_{\Rdft} \to \CC, z
  \mapsto \sum_{j
  \in \NN_0} c_j z^j$
  be analytic, and let $\ab(\tilde{c}_{j, \rdft})_{j =
  0}^{\ndft - 1}$ be the discrete Fourier coefficients of $f$
  using $\ndft \in \NN$ sampling points and radius $\rdft > 0$, defined as
  \begin{equation}
    \tilde{c}_{j,\rdft} \coloneq \frac{1}{\ndft \rdft^j} \sum_{j' =
    0}^{\ndft-1} f\ab(\rdft \exp\ab(\frac{2 \pi i j'}{\ndft}))
    \exp\ab(-\frac{2 \pi i j j'}{\ndft}).
    \label{dft-definition}
  \end{equation}
  Let $M_{\Rdft} \coloneq \max_{\abs{z} = \Rdft} \abs{f(z)}$.
  Then for all $j \in \ab\{0, \dots, \ndft - 1\}$,
  \[
    \abs{\tilde{c}_{j,\rdft} - c_j} \le M_{\Rdft} \Rdft^{-j} \frac{\ab(\rdft /
    \Rdft)^{\ndft}}{1 - \ab(\rdft / \Rdft)^{\ndft}}.
  \]
\end{theorem}
\begin{proof}
  This follows from the error estimate
  in~\parencite[(3.5)]{niu_verified_2007} applied to $z \mapsto f(\rdft
  z)$ with Cauchy's estimate on the disk of radius $\Rdft / \rdft$.
\end{proof}

\section{Preconditioning}
\label{chap-preconditioning}

For the case of scalar (or $1 \times 1$ matrix-valued) functions,
the following \dlg\ method,
which raises the roots of a function to the $G$-th power
for any specified $G \in \NN$, is applied to guarantee the existence
of Pellet bounds and further improve them.

\begin{theorem}[\dlg~\parencite{householder_dandelin_1959}]
  Let $\rho > 0$, let $f: D_\rho \to \CC$ be analytic and $f
  \not\equiv 0$, and
  let $G \in \NN$.
  Let $\omega \coloneq \exp\ab(2 \pi i/G)$, let $\gr^G (f) (z^G) \coloneq
  \prod_{m = 0}^{G - 1} f(\omega^{m} z)$.
  Let $\ab(Z_j)_{j = 1}^p$ be the roots of $f$ in $D_\rho$,
  counting multiplicities.
  Then $\gr^G (f): D_{\rho^G} \to \CC$ is analytic and has
  roots $\ab(Z_j^G)_{j=1}^p$, counting multiplicities.
\end{theorem}

It follows directly from this definition that $\gr^G (1) = 1$
and that $\gr^G$ is a
multiplicative homomorphism, i.e., $\gr^G (f g) = \gr^G
(f) \gr^G (g)$.

In the remainder of this section, the convergence of Pellet bounds
after applying the \dlg\ method is proved,
although these results are not necessarily useful for the
numerical computation of Pellet bounds themselves.

To prove the convergence, it needs to be shown that after applying
the \dlg\ method with sufficiently large
order $G$, the $k$-th term of $\pl_k$ or $\fpc_k^n$
becomes dominant in
some region (which might be the Pellet bound), for the specified $k$ and $n$.
Hence, each coefficient of $\gr^G f$ needs to be roughly estimated.
To do so, instead of directly estimating the coefficients of $\gr^G
f$, the coefficients of $\gr^G (\exp \circ f)$ are estimated using
the following property:

\begin{lemma}
  \label{graeffe-log}
  Let $\rho > 0$, let $f: D_\rho \to \CC, z \mapsto \sum_{j \in
  \NN_0} c_j z^j$
  be analytic. Then $\gr^G (\exp \circ f) (z) = \exp\ab(G \sum_{j
  = 0}^\infty c_{G j} z^{j})$.
\end{lemma}
\begin{proof}
  By definition, one has
  \begin{align*}
    \gr^G (\exp \circ f) (z^G) &= \prod_{m = 0}^{G - 1}
    \exp\ab(f(\omega^{m} z))
    = \exp\ab(\sum_{m = 0}^{G - 1} f(\omega^{m} z))
    = \exp\ab(\sum_{m = 0}^{G - 1} \sum_{j = 0}^\infty c_j
    (\omega^{m} z)^j) \\
    &= \exp\ab(\sum_{j = 0}^\infty c_j z^j \sum_{m = 0}^{G - 1}
    \omega^{j m})
    = \exp\ab(G \sum_{j = 0}^\infty c_{G j} z^{G j}).
  \end{align*}
\end{proof}

The following lemmas are also used in the proof of convergence.

\begin{lemma}
  \label{exp-inequality}
  Let $\alpha > 0$, $\beta > 0$.
  If $x > 2/(\beta \alpha^2)$, then $x \exp \ab(-\alpha x) < \beta$.
\end{lemma}
\begin{proof}
  Since $\exp\ab(\alpha x) > (\alpha x)^2 / 2$, $x \exp\ab(-\alpha x) < 2 /
  (\alpha^2 x)$ holds. Therefore, if $2/(\alpha^2 x) < \beta$, $x
  \exp\ab(-\alpha x) < \beta$ holds.
\end{proof}

\begin{lemma}
  \label{geometric-series-inequality}
  Let $r \in (0, 1)$, $n \in \NN$.
  If $n > (\log 2)/(-\log r)$, then $1/(1-r^n) < 2$.
\end{lemma}
\begin{proof}
  If $n > (\log 2)/(-\log r)$, $n \log r < - \log 2$ holds.
  By taking the exponential of both sides, $r^n < 1/2$ holds, and $1
  - r^n > 1/2$ holds.
\end{proof}

In what follows, the $j$-th coefficient of an analytic function
$f$ is denoted
by $\coef{f}_j$. The coefficients of $\gr^G f$ are estimated as follows.

\begin{lemma}
  \label{gr-coef-bounds}
  Let $\xh > 0$, let $f: D_{\xh} \to \CC, z \mapsto \sum_{j \in
  \NN_0} c_j z^j \not\equiv 0$ be analytic. Assume further that
  $f$ has exactly $p \in \NN_0$ roots $\ab(Z_j)_{j = 1}^p$
  in $D_{\xh}$, counting
  multiplicities, where the roots are ordered
  as $\abs{Z_1} \le \abs{Z_2} \le \dots \le \abs{Z_p}$. Let
  \begin{equation*}
    \xmin \coloneq
    \begin{cases} \abs{Z_p} & (p \in \NN), \\ 0 & (p = 0).
    \end{cases}
  \end{equation*}
  Let $\xeval, \xc$ satisfy $\xmin < \xeval < \xc < \xh$.
  Let $\delta \coloneq \max\ab(\xmin / \xeval, \xeval / \xc, \xc / \xh) \in
  (0, 1)$, let $L \coloneq (-\log \delta)^{-1}$.
  Let $q(z) \coloneq \prod_{j = 1}^p (z - Z_j), h(z) \coloneq f(z)/q(z)$. Let $G
  \in \NN$, let $\Kg \coloneq \abs{h(0)}^G$, and let
  $M_{\xh} \coloneq
  \max_{\abs{z} = \xh} \abs{\log h(z)}$.
  Then for sufficiently large $G$,
  more precisely if $G > \max\ab(L (p + 1) \log 2, (4 M_{\xh}
  L^2)/(\log 2))$,
  the following bounds hold:
  \begin{enumerate}[label=(\arabic*)]
    \item (Lower bound for $j = p$). $\abs{\coef{\gr^G f}_p} \Kg^{-1}
      \ge 1 - 2^{p + 1} \delta^{G}.$
    \item (Upper bounds). $\abs{\coef{\gr^G f}_j} \xeval^{G(j - p)}
      \Kg^{-1} \le 2^{p + 1} (p + 1) \delta^{G \abs{j - p}},
      \quad j \in \NN_0.$
  \end{enumerate}
\end{lemma}
\begin{proof}
  The proof is organized as follows:
  $f$ is split into $q$ and $h$, and the coefficients of $\gr^G q$
  and $\gr^G h$ are estimated separately, then combined to estimate
  the coefficients of $\gr^G f$.
  \begin{enumerate}[label=\arabic*.]
    \item (Bounds for $\coef{\gr^G q}$).
      \begin{enumerate}[label=(\alph*)]
        \item (Exact value for $j = p$). Since $q$ is a monic
          polynomial of degree $p$,
          \begin{equation}
            \abs{\coef{\gr^G q}_p} = 1.
            \label{gr-q-coef-p}
          \end{equation}
        \item (Upper bound). Vieta's formulas imply
          \begin{equation}
            \abs{\coef{\gr^G q}_j} \le \binom{p}{j} \xmin^{G(p - j)},
            \quad j \in \{0, \dots, p\}.
            \label{gr-q-coef-j}
          \end{equation}
      \end{enumerate}
    \item (Upper bound for $\coef{H}$).
      Let $H(z) \coloneq \log h(z)$.
      Since $h$ has no roots in $\abs{z} \le \xh$, $H$ is analytic in
      $\abs{z} \le \xh$. Thus \zcref{cauchy-estimate} implies
      \begin{equation}
        \abs{\coef{H}_j} \le M_{\xh} \xh^{-j}, \quad j \in \NN_0.
        \label{H-coefficients}
      \end{equation}
    \item (Bounds for $\coef{\gr^G h}$).
      \begin{enumerate}[label=(\alph*)]
        \item (Exact value for $j = 0$). Since $\coef{\gr^G
          h}_0 = (\gr^G h)(0) = h(0)^G$, $\abs{\coef{\gr^G h}_0} = \Kg$ holds.
        \item (Upper bounds).
          \zcref{cauchy-estimate} again implies
          \begin{equation}
            \abs{\coef{\gr^G h}_j} \le M_{\xc}^{(G)} \xc^{- G j},
            \quad j \in \NN_0,
          \end{equation}
          where
          \[
            M_{\xc}^{(G)} \coloneq \max_{\abs{z} = \xc^G}
            \abs{\gr^G h(z)}.
          \]
          \zcref{graeffe-log} and \zcref{H-coefficients} imply
          \begin{align*}
            M_{\xc}^{(G)}
            &= \max_{\abs{z} = \xc^G} \abs{\exp\ab(G \ab(\coef{H}_0 +
            \sum_{j = 1}^\infty \coef{H}_{G j} z^{j}))} \\
            &\le \Kg \exp\ab(M_{\xh} G \sum_{j = 1}^\infty \xh^{-G
            j} \xc^{G j}) \\
            &\le \Kg \exp\ab(\frac{G M_{\xh} \delta^G}{1 - \delta^G}).
          \end{align*}
          \zcref{geometric-series-inequality} implies that if $G > L \log
          {2}$, $M_{\xc}^{(G)} \le \Kg \exp\ab(2 G M_{\xh} \delta^G)$
          holds. If $M_{\xh} \neq 0$, since $\delta^G = \exp\ab(-G / L)$,
          \zcref{exp-inequality} implies that by further choosing $G >
          (4 M_{\xh} L^2)/(\log 2)$, $G \exp(-G/L) < (\log 2)/(2
          M_{\xh})$ and $2 G M_{\xh} \delta^G < \log 2$
          hold and thus $M_{\xc}^{(G)} \le 2 \Kg$ holds.
          Therefore, one has
          \begin{equation}
            \abs{\coef{\gr^G h}_j} \le 2 \Kg \xc^{- G j}, \quad
            j \in \NN_0.
            \label{gr-h-coef-j}
          \end{equation}
          If $M_{\xh} = 0$, \zcref{gr-h-coef-j} holds trivially.
      \end{enumerate}
    \item (Bounds for $\coef{\gr^G f}$).
      Since $\gr^G$ is a multiplicative homomorphism,
      $\gr^G f = \gr^G q\, \gr^G h$ holds. Therefore,
      \begin{equation}
        \coef{\gr^G f}_j = \sum_{m = 0}^{\min(j, p)} \coef{\gr^G q}_m
        \coef{\gr^G h}_{j - m}.
        \label{gr-f-by-gr-q-gr-h}
      \end{equation}
      \begin{enumerate}[label=(\alph*)]
        \item (Lower bound for $j = p$).
          The reverse triangle inequality, \zcref{gr-q-coef-p},
          \zcref{gr-q-coef-j} and \zcref{gr-f-by-gr-q-gr-h} imply
          \begin{align*}
            \abs{\coef{\gr^G f}_p} \Kg^{-1}
            &\ge \ab(\abs{\coef{\gr^G q}_p} \abs{\coef{\gr^G h}_0}
              - \sum_{j = 0}^{p - 1} \abs{\coef{\gr^G q}_j}
            \abs{\coef{\gr^G h}_{p - j}})\Kg^{-1} \\
            &\ge 1 - 2 \sum_{j = 0}^{p - 1} \binom{p}{j} \xmin^{G(p -
            j)} \xc^{-G(p - j)} \\
            &\ge 1 - 2^{p + 1} \delta^{G},
          \end{align*}
          where the last expression, $1 - 2^{p + 1} \delta^{G}$,
          is positive if $G > (p + 1)
          L \log 2$ holds additionally.
        \item (Upper bounds). By \zcref{gr-q-coef-j},
          \zcref{gr-h-coef-j} and \zcref{gr-f-by-gr-q-gr-h}, one has
          \begin{align*}
            \abs{\coef{\gr^G f}_j} \xeval^{G(j - p)} \Kg^{-1}
            &\le \sum_{m = 0}^{\min(j, p)} \abs{\coef{\gr^G q}_m}
            \abs{\coef{\gr^G h}_{j - m}} \xeval^{G(j - p)} \Kg^{-1} \\
            &\le 2^{p + 1} \sum_{m = 0}^{\min(j, p)} \xmin^{G(p - m)}
            \xc^{G (m - j)} \xeval^{G(j - p)} \\
            &= 2^{p + 1} \sum_{m = 0}^{\min(j, p)}
            (\xmin/\xeval)^{G(p - m)} (\xeval/\xc)^{G (j - m)} \\
            &\le 2^{p + 1} (p + 1) \delta^{G \abs{j - p}}.
          \end{align*}
      \end{enumerate}
  \end{enumerate}
\end{proof}

\begin{theorem}
  \label{graeffe-convergence}
  Let $\rho > 0$, let $f: D_\rho \to \CC, z \mapsto \sum_{j \in
  \NN_0} c_j z^j \not\equiv 0$ be analytic.
  Further assume that $f$ has exactly $p \in \NN_0$ roots
  $\ab(Z_j)_{j = 1}^p$ in $D_\rho$, counting
  multiplicities, where the roots are ordered
  as $\abs{Z_1} \le \abs{Z_2} \le \dots \le \abs{Z_p}$.
  Let $k \in \{0, ..., p - 1\}$, and let $\xeval$ be such that
  $\xmin \coloneq
  \begin{cases} \abs{Z_k} & (k \in \NN), \\ 0 & (k = 0)
  \end{cases} < \xeval < \abs{Z_{k + 1}} =: \xmax$.
  Let $G \in \NN$. Let $\pl_k, \pc_k^n$ be the functions for
  $\gr^G f$ defined as in \zcref{generalized-pellet-infinite} and
  \zcref{generalized-pellet-cauchy}.
  Then the following assertions hold:
  \begin{enumerate}[label=(\arabic*)]
    \item If $G$ is large enough,
      more precisely if
      \[
        G > \max\ab((4 M_{\xh} L^2)/(\log 2), L \log\ab(2^{k + 5}
        (k + 1))),
      \]
      where $\delta \coloneq \max\ab(\xmin/\xeval, (\xeval/\xmax)^{1/3}) < 1$,
      $L \coloneq (-\log \delta)^{-1}$, $\xh \coloneq
      \xeval^{1/3} \xmax^{2/3}$,
      $M_{\xh} \coloneq \max_{\abs{z} = \xh}
      \abs{\log h(z)}$ and $h(z) \coloneq f(z)/(\prod_{j = 1}^{k}
      (z - Z_j))$,
      then $\pl_k (\xeval^G) < 0$.

    \item For any $G$ satisfying the condition in (1), for any
      $M_{\rho}^{(G)} \ge \max_{\abs{z} = \rho^G}
      \abs{(\gr^G f)(z)}$, and for sufficiently large $n$
      (depending on $G$ and $M_{\rho}^{(G)}$),
      $\pc_k^n (\xeval^G) < 0$ holds.
  \end{enumerate}
\end{theorem}
\begin{proof}
  Let $\xc \coloneq \xeval^{2/3}
  \xmax^{1/3}$, $\delta \coloneq \max\ab(\xmin/\xeval, \xeval/\xc, \xc/\xh,
  \xh/\xmax) = \max\ab(\xmin/\xeval, \sqrt[3]{\xeval/\xmax}) < 1$.

  \begin{enumerate}[label=(\arabic*)]
    \item \zcref{gr-coef-bounds} with $p = k$ implies
      \begin{align*}
        \frac{\pl_k (\xeval^G)}{\xeval^{G k} \Kg}
        &= \sum_{j \in \NN_0 \setminus \{k\}} \frac{\abs{\coef{\gr^G
        f}_j} \xeval^{G(j - k)}}{\Kg} - \frac{\abs{\coef{\gr^G f}_k}}{\Kg} \\
        &\le 2^{k + 1} \ab((k + 1) \sum_{j \in \NN_0 \setminus
        \{k\}} \delta^{G \abs{j - k}} + \delta^G) - 1 \\
        &\le 2^{k + 3} (k + 1) \sum_{j = 1}^\infty \delta^{G j} - 1 \\
        &= 2^{k + 3} (k + 1) \frac{\delta^G}{1 - \delta^G} - 1 \\
        &\le 2^{k + 4} (k + 1) \delta^G - 1 \\
        &< -\frac{1}{2}
      \end{align*}
      if $G > L \log\ab(2^{k + 5} (k + 1))$ additionally. This completes
      the proof for $\pl_k$.
    \item By definition $\pt_k^n (x) \le \pl_k (x)$ holds. Therefore,
      \[
        \fpc_k^n (x) = \pt_k^n (x) + M_{\rho}^{(G)}
        \frac{(x/\rho^G)^n}{1 - x/\rho^G} \le \pl_k (x) +
        M_{\rho}^{(G)} \frac{(x/\rho^G)^n}{1 - x/\rho^G}
      \]
      holds. Since $\xeval < \rho$, one has
      \begin{align*}
        \frac{\fpc_k^n (\xeval^G)}{\xeval^{G k} \Kg}
        &< -\frac{1}{2} + \frac{M_{\rho}^{(G)}}{\xeval^{G k} \Kg}
        \frac{(\xeval^G/\rho^G)^n}{1 - \xeval^G/\rho^G} \\
        &\le -\frac{1}{2} + \frac{2 M_{\rho}^{(G)}}{\xeval^{G k} \Kg}
        (\delta^G)^n \\
        &< 0
      \end{align*}
      if $n > L \ab((\log 4 + \log M_\rho^{(G)})/G - k
      \log \xeval - \log \abs{h(0)})$ holds.
      This completes the proof
      for $\pc_k^n$.
  \end{enumerate}
\end{proof}

The values $\ab(\coef{\gr^G f}_j)_{j = 0}^{n -1}$ may be computed
by applying
\zcref{ck-error} directly to $\gr^G f$. However, it is more efficient
to restrict $G$ to be a power of two,
compute $\ab(c_j)_{j = 0}^{G (n - 1)}$ by applying \zcref{ck-error} to
$f$, and then apply the following theorem $\log_2 G$ times.

\begin{theorem}[\dlg\ iteration~\parencite{householder_dandelin_1959}]
  \label{graeffe-iteration}
  Let $\rho > 0$, let $f: D_\rho \to \CC, z \mapsto \sum_{j \in
  \NN_0} \coef{f}_j
  z^j$ be analytic, let $k \in \NN_0$. Then
  \[
    \coef{\gr^2 f}_k = (-1)^k \ab(\coef{f}_k)^2 + 2 \sum_{j = 0}^{k - 1}
    (-1)^j \coef{f}_j \coef{f}_{2 k - j}.
  \]
\end{theorem}

We need to choose the radii $\rdft, \Rdft$.
\zcref{ck-error} implies that choosing $\Rdft \gg \rdft$ seems better,
while it still does not guide the choice of their scales.
However, assuming that constant relative rounding error is introduced
when evaluating $f$, it may be seen from \zcref{dft-definition} that
the higher-order coefficients of $f$ are prone to contamination if
$\rdft$ is small (especially for $\rdft \le 1$).
Therefore, when computing $\coef{\gr^2 f}$ using
\zcref{graeffe-iteration}, in order to avoid the propagation of the
relative error of higher-order coefficients of $f$, $\rdft$ needs to
be chosen large enough depending on the decay rate of the coefficients of $f$.

\section{Numerical examples}
\label{chap-numerical-examples}

The interval enclosures of the roots of $\pc_k^n$ in
\zcref{generalized-pellet-cauchy} are calculated using an interval
arithmetic package, python-flint~\parencite{johansson_arb_2017}.
The precision for the mantissa (\ctxprec) is configured per example.
The upper bound $M_\rho$ is obtained by splitting the circle
$\abs{z} = \rho$ into multiple intervals and evaluating the upper
bound of $\norm{F(z) - \sum_{j = 0}^{n-1} C_j z^j}$ on each interval,
then taking the maximum of these values.
More sophisticated methods are proposed in
\textcite{neher_validated_2001, eble_acetaf_2003}.

\begin{example}
  Let $f(z) \coloneq \sin \pi z = \pi z \prod_{m \in \NN} (1 -
  z^2/m^2)$.
  Then the function $f$ has roots at each $z \in \ZZ$, and all
  of them have multiplicity one.
  The parameters are set to $\rho = 10, \ndft=16384$, $\rdft = 4$
  and $\Rdft=100$.
  The Pellet bounds for $\gr^G f$ are shown in
  \zcref{table-sin-n}, \zcref{table-sin-graeffe}, \zcref{table-sin-k}.
  Only odd values of $k$ are considered, since no Pellet bounds
  can be obtained for even $k$ due to the symmetry of the roots of $f(z)$.
  Here, the lower bounds are rounded up to five decimal places,
  and the upper bounds are rounded down to five decimal places.
  From \zcref{table-sin-n} and \zcref{table-sin-graeffe}, one can see
  that increasing $G$ and $n$ improves the Pellet bounds, while
  from \zcref{table-sin-k}, one can see that the Pellet
  bounds become worse as $k$ increases for fixed $G$ and $n$. While
  \zcref{graeffe-convergence} provides sufficient but not
  necessary conditions on $n$
  and $G$, this behavior looks consistent with the theorem.
\end{example}

\ifpreprint\begin{table}[ht]\else\begin{table}[pos=ht]\fi
  \caption{Pellet bounds for $\sin \pi z$ with $k = 1$, $G =
  1$, $\ctxprec=64$ and $n \in \{4,8,16,32\}$.}
  \centering
  \begin{filecontents*}{table1.csv}
n,start,end
3,0.00001,0.00111
7,0.00001,0.19975
15,0.00001,0.69306
31,0.00001,0.69306
\end{filecontents*}
\csvreader[
    tabular=ccc,
    table head=\toprule $n$ & lower bound & upper bound \\ \midrule,
    table foot=\bottomrule,
  ]{table1.csv}{}{
  \csvcoli & \csvcolii & \csvcoliii}
  \label{table-sin-n}
\end{table}

\ifpreprint\begin{table}[ht]\else\begin{table}[pos=ht]\fi
  \caption{Pellet bounds for $\sin \pi z$ with $k = 3$, $n =
    32$, $\ctxprec=1024$ and $G \in \{4, 8, 16, 32\}$.
  Pellet bounds were not obtained for $G = 4$.}
  \centering
  \begin{filecontents*}{table2.csv}
G,start,end
4,not found,not found
8,1.11732,1.77671
16,1.05664,1.89254
32,1.02793,1.94566
\end{filecontents*}
\csvreader[
    tabular=ccc,
    table head=\toprule $G$ & lower bound & upper bound \\ \midrule,
    table foot=\bottomrule,
  ]{table2.csv}{}{
  \csvcoli & \csvcolii & \csvcoliii}
  \label{table-sin-graeffe}
\end{table}

\ifpreprint\begin{table}[ht]\else\begin{table}[pos=ht]\fi
  \caption{Pellet bounds for $\sin \pi z$ with $G = 128$, $n =
  32$, $\ctxprec=1024$ and $k \in \{1, 3, 5, 7\}$.}
  \centering
  \begin{filecontents*}{table3.csv}
k,start,end
1,0.00001,0.99313
3,1.00691,1.98627
5,2.01382,2.97941
7,3.02073,3.97255
\end{filecontents*}
\csvreader[
    tabular=ccc,
    table head=\toprule $k$ & lower bound & upper bound \\ \midrule,
    table foot=\bottomrule,
  ]{table3.csv}{}{
  \csvcoli & \csvcolii & \csvcoliii}
  \label{table-sin-k}
\end{table}

\begin{example}[{\normalfont\texttt{hadeler}}
  problem~\parencite{hadeler_mehrparametrige_1967, betcke_nlevp_2013}]
  \label{hadeler-problem}
  Let $l \in \NN$, $\alpha > 0$.
  The Hadeler problem is the problem of finding $z \in \CC$ such
  that $\det F(z) = 0$, where
  \[
    F(z) := \ab((\exp z - 1) B + z^2 A_2 - A_0)
  \]
  and
  \[
    (B)_{p q} \coloneq (l + 1 - \max\ab(p,
    q)) p q, \quad (A_2)_{p q} \coloneq l
    \delta_{p q} + \frac{1}{p + q}, \quad A_0
    \coloneq \alpha I_{l}.
  \]

  The default parameters in NLEVP~\parencite{betcke_nlevp_2013}, $l =
  8$ and $\alpha = 100$,
  were used.
  The function
  $F$ was reimplemented using python-flint to support
  interval arithmetic.
  The parameters were set to $\rho = 10, \ndft=8192$, $\rdft =
  4$ and $\Rdft=100$.
  The 1-norm was used as the matrix norm.
  The obtained Pellet bounds for $F(c + \cdot)$ with $k \in \{0, 1\}$ and
  $\gr^G \det F(c + \cdot)$ with $G = 2^3$ and $k \in \{0,1,2,3,4\}$, with
  centers $c \in \ab\{x + y i \mid x, y \in \{-2, 0, 2\}\}$,
  $\ctxprec=128$ and $n=50$ are shown in
  \zcref{fig-hadeler-bounds}.
  For $F(c + \cdot)$ with $k = 1$, no Pellet bounds were
  obtained for any $c$, while for $\gr^G \det F(c + \cdot)$ with $k \in
  \{1, 2,3,4\}$, Pellet bounds were successfully obtained for some $c$.
  The Pellet bounds for $\gr^G \det F(c + \cdot)$ were tighter than
  the Pellet bounds for $F(c + \cdot)$.
\end{example}

\ifpreprint\begin{figure}[ht]\else\begin{figure}[pos=ht]\fi
  \centering
  \includegraphics[width=13cm]{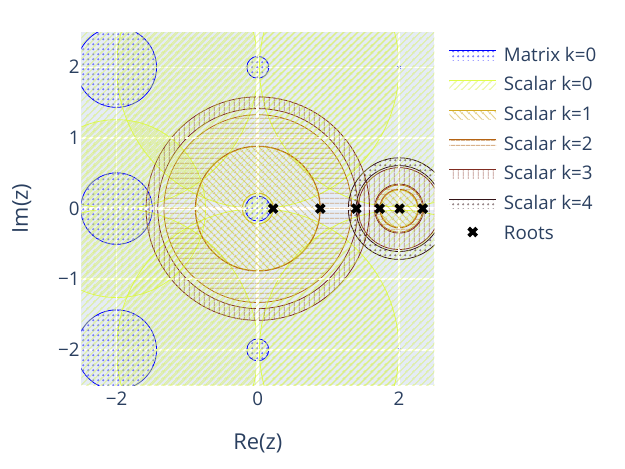}
  \caption{Pellet bounds for the Hadeler problem with $l = 8$
    and $\alpha = 100$ plotted in $\{x + y i \mid x, y \in [-2.5, 2.5]\}$.
    Blue dotted regions are the Pellet bounds obtained by
    matrix Pellet's theorem with Cauchy's estimate for $k = 0$.
    Other regions are the Pellet bounds obtained by scalar
    Pellet's theorem with Cauchy's estimate for $\gr^G (\det F)$ for
  $G = 2^3$ and $k \in \{0, 1, 2, 3, 4\}$.}
  \label{fig-hadeler-bounds}
\end{figure}

\section*{Acknowledgements}

The authors would like to thank Dr.\ Atsushi Nakayasu for their
valuable comments.
This work is supported by "Joint Usage/Research Center for
Interdisciplinary Large-scale Information Infrastructures (JHPCN)"
and "High Performance Computing Infrastructure (HPCI)" in Japan
(Project ID: jh260048).
This research was conducted using the Supermicro ARS-111GL-DNHR-LCC
and FUJITSU Server PRIMERGY CX2550 M7 (Miyabi) at Joint Center for
Advanced High Performance Computing (JCAHPC).
This work was supported in part by JSPS KAKENHI Grant Number JP23H03800 and 25H01108, and JST FOREST Program Grant Number JPMJFR202J.

\section*{Data availability}

The code used to conduct this study is available
at~\parencite{itoh_2026_23094309}.

\section*{Declaration of generative AI and AI-assisted technologies
in the manuscript preparation process}

During the preparation of this work, the authors used GitHub Copilot
and OpenCode Go in
order to correct grammatical errors and convert the manuscript to
LaTeX. After using these services, the authors reviewed and edited the
content as needed and take full responsibility for the content of the
published article.

\printbibliography

\end{document}